\documentclass[letterpaper, 10 pt, conference]{ieeeconf}  

\IEEEoverridecommandlockouts                              
\usepackage{graphics} 
\usepackage{graphicx}
\usepackage{mathptmx} 
\usepackage{times} 
\usepackage{amsmath} 
\usepackage{amssymb}  
\newtheorem{theorem}{Theorem}
\newtheorem{lemma}{Lemma}       
\newtheorem{corollary}{Corollary} 
\newtheorem{proposition}{Proposition} 
\newtheorem{assumption}{Assumption}
\newtheorem{definition}{Definition}

\usepackage{algorithm}
\usepackage{algpseudocode}
\usepackage{multirow}
\usepackage{booktabs}
\usepackage{bm} 
\usepackage{dsfont}

\usepackage{amsfonts}
\DeclareMathAlphabet{\mathcal}{OMS}{cmsy}{m}{n}

\title{\LARGE \bf
Exact Learning of Linear Model Predictive Control Laws \\ using Oblique Decision Trees with Linear Predictions
}

\author{ \parbox{4 in}{\centering Jiayang Ren, Qiangqiang Mao, Tianwei Zhao, Yankai Cao$^\dagger$	\\
        Department of Chemical and Biological Engineering\\
        University of British Columbia, Vancouver, Canada\\
        { $^\dagger$ \tt\small yankai.cao@ubc.ca}}
}

\begin{document}
\maketitle
\thispagestyle{empty}
\pagestyle{empty}

\begin{abstract}
Model Predictive Control (MPC) is a powerful strategy for constrained multivariable systems but faces computational challenges in real-time deployment due to its online optimization requirements. While explicit MPC and neural network approximations mitigate this burden, they suffer from scalability issues or lack interpretability, limiting their applicability in safety-critical systems. This work introduces a data-driven framework that directly learns the Linear MPC control law from sampled state-action pairs using Oblique Decision Trees with Linear Predictions (ODT-LP), achieving both computational efficiency and interpretability. By leveraging the piecewise affine structure of Linear MPC, we prove that the Linear MPC control law can be replicated by finite-depth ODT-LP models. We develop a gradient-based training algorithm using smooth approximations of tree routing functions to learn this structure from grid-sampled Linear MPC solutions, enabling end-to-end optimization. Input-to-state stability is established under bounded approximation errors, with explicit error decomposition into learning inaccuracies and sampling errors to inform model design. Numerical experiments demonstrate that ODT-LP controllers match MPC's closed-loop performance while reducing online evaluation time by orders of magnitude compared to MPC, explicit MPC, neural network, and random forest counterparts. The transparent tree structure enables formal verification of control logic, bridging the gap between computational efficiency and certifiable reliability for safety-critical systems.
\end{abstract}

\section{Introduction}\label{sec:intro}

Model Predictive Control (MPC) is an optimization-based control strategy widely employed for regulating multivariable systems under operational constraints. However, its practical implementation requires solving an optimization problem at each time step, rendering it computationally demanding for real-time applications. This challenge is particularly acute in systems with fast dynamics or high-dimensional models. While advances in optimization solvers have improved the feasibility of MPC implementations \cite{kochProgress2022}, computational complexity remains a critical bottleneck in many scenarios.

To address these limitations, explicit MPC \cite{bemporadModelPredictiveControl2002, bemporad_explicit_2002, alessioSurveyExplicitModel2009, summers_multi_2011, bemporadMPQP2015} represents a prominent approach for reducing online computation. By precomputing the control law offline via multiparametric programming, explicit MPC expresses the solution as a piecewise function of system states, reducing online execution to a region-identification task. Despite eliminating real-time optimization, this method suffers from exponential growth in offline computational and memory requirements, which scale exponentially with the prediction horizon and the number of states. Consequently, explicit MPC is typically applicable only to small-scale systems or problems with limited complexity.

For large-scale systems, Neural Networks (\texttt{NN}) have emerged as an alternative for approximating MPC policies \cite{parisiniRecedinghorizon1995, hertneckLearning2018, kumar_industrial_2021, okamoto_deep_2024}. Leveraging their universal approximation capabilities \cite{hornikMultilayer1989}, \texttt{NN}s learn mappings between system states and optimal control inputs, enabling rapid online inference. However, the black-box nature of \texttt{NN}s compromises transparency, hindering verification of reliability, safety, and decision logic \cite{rudin_stop_2019}. These limitations pose significant risks in safety-critical domains such as autonomous systems and chemical plants, where interpretability is paramount for compliance and operational trust.

\begin{figure}
    \centering
    \includegraphics[width=0.8\columnwidth]{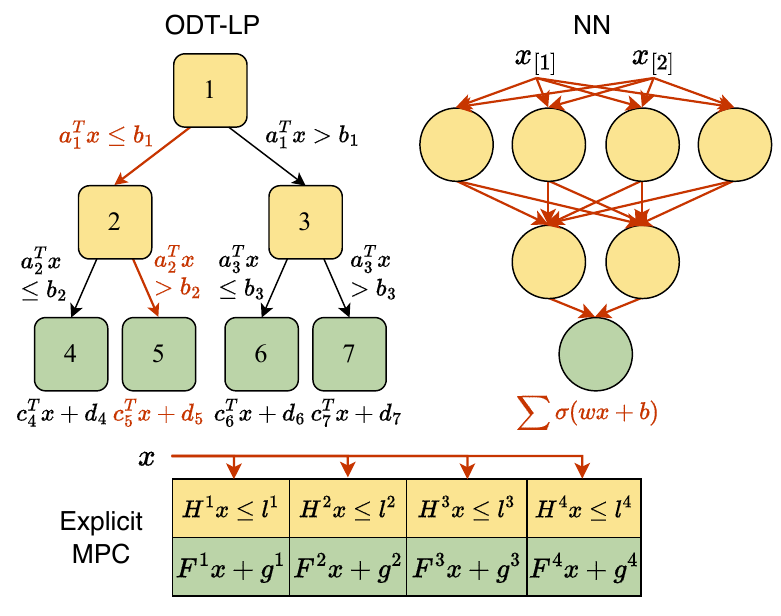}
    \vskip -0.1in
    \caption{Examples of \texttt{ODT-LP}, \texttt{NN}, and Explicit MPC control laws. (Red paths show the activated paths needed to obtain actions. ODT-LP visits one path, NN visits all paths, and explicit MPC visits all regions in worst cases.)}
    \label{fig:dt}
    \vskip -0.2in
\end{figure}

To bridge this gap, we propose a data-driven framework that directly learns the MPC control law from sampled state-action pairs using oblique decision trees with linear predictions (\texttt{ODT-LP}), eliminating the need for explicit form computation. As shown in Fig.~\ref{fig:dt}, \texttt{ODT-LP} partitions the state space using interpretable rules and employs linear models at leaf nodes for predictions, ensuring transparency in control logic. Comparing to \texttt{NN}s with equivalent parameters, \texttt{ODT-LP} achieves faster inference by activating only one path (red paths in Fig.~\ref{fig:dt}) to the leaf node, bypassing full-path computation in \texttt{NN}s. The main contributions are:  
\begin{itemize}
    \item A demonstration that Linear MPC control laws can be exactly represented by \texttt{ODT-LP};  
    \item A data-driven learning framework for \texttt{ODT-LP}-based MPC that uses grid-sampled state-action datasets and gradient-based optimization to train the decision tree.
    \item Numerical validations shows that \texttt{ODT-LP} can exactly learn the Linear MPC control law, achieve faster online implementation than explicit MPC and \texttt{NN}-based MPC, and maintain interpretability in the control logic.
\end{itemize}

\section{Preliminaries and Notation}
\subsection{Oblique Decision Tree with Linear Predictions}
Decision trees organize decisions and their outcomes into a tree-like structure. As shown in Fig. \ref{fig:dt}, the decision tree recursively partitions the dataset based on the splitting rule at each branch node (yellow nodes in the figure), making predictions at each leaf node (green nodes). In this paper, we focus on the Oblique Decision Tree with Linear Predictions, where the split rule at each branch node is a linear combination of features, and the prediction at each leaf node is computed by a linear model.

Formally, consider a dataset $(\mathbf{X}, \mathbf{Y})$ with $S$ samples, where each sample includes a feature vector $\mathbf{x}_i \in \mathbb{R}^n$ and a label vector $\mathbf{y}_i \in \mathbb{R}^m$.
An \texttt{ODT-LP} with depth $D$ comprises $T = 2^{(D+1)} - 1$ nodes. Branch nodes, indexed by $\mathcal{T}_B = \{1, \ldots, \lfloor T/2 \rfloor\}$, handle splits, while leaf nodes, indexed by $\mathcal{T}_L = \{\lfloor T/2 \rfloor + 1, \ldots, T\}$, provide predictions.

Each branch node $t \in \mathcal{T}_B$ is characterized by a split vector ${a}_t \in [0, 1]^p$ and a threshold $b_t \in [0, 1]$. For a sample $\mathbf{x}_i$, the split rule is: $ \mathbf{x}_i \text{ proceeds to the left child node if } {a}_t^\top \mathbf{x}_i \leq b_t, \text{ otherwise to the right child node.}$

Each leaf node $t \in \mathcal{T}_L$ is associated with a prediction model. When a sample $\mathbf{x}_i$ reaches $t$th leaf node following the split rules $(\mathbf{a}, \mathbf{b})$ in the branch nodes, the predicted value of this sample is computed as: $\hat{\mathbf{y}}_i = f_{\text{tree}}(\mathbf{a}, \mathbf{b}, \mathbf{c}, \mathbf{d}, \mathbf{x}_i) = {c}_t^\top \mathbf{x}_i + d_t$, where ${c}_t$ and $d_t$ are the parameters of $t$th leaf node's prediction model, learned to minimize the prediction error of samples assigned to this leaf node.

The optimization objective for training the tree minimizes the prediction error:
\vskip -0.2in
\begin{equation}
    \begin{aligned}
        & \min_{\mathbf{a}, \mathbf{b}, \mathbf{c}, \mathbf{d}} && \quad \sum_{i=1}^S \|\hat{\mathbf{y}}_i - \mathbf{y}_i\|_2^2,\\
        & \text{s.t. } && \quad  \hat{\mathbf{y}}_i = f_{\text{tree}}(\mathbf{a}, \mathbf{b}, \mathbf{c}, \mathbf{d}, \mathbf{x}_i).
    \end{aligned}
\end{equation}
\subsection{Linear Model Predictive Control}
Considering a discrete-time linear time-invariant system as follows:
\begin{equation} \label{eqn:sys}
    \begin{aligned}
        & x(k+1) = Ax(k) + Bu(k),\\
        & y(k) = Cx(k) + Du(k),
    \end{aligned}
\end{equation}
where $k \in \mathbb{I}_{\geq 0}$ is a nonnegative integer denoting the sample number, and time can be referred to as $t = k\Delta$, where $\Delta$ is the sample time. $x(k) \in \mathbb{R}^n$ is the state vector, $u(k) \in \mathbb{R}^m$ is the input vector, and $y(k) \in \mathbb{R}^p$ is the output vector. It is assumed that the pair $(A, B)$ is stabilizable and that a full measurement of the state $x(k)$ is available at the current sampling time $k$.

Model predictive control problem aims to regulate system to the origin by solving the following optimization problem:
\begin{equation} \label{eqn:mpc}
    \begin{aligned}
        & \min_{\mathbf{u}} && \sum_{k=0}^{N-1} (x(k)^\top Qx(k) + u(k)^\top Ru(k)) + x(N)^\top Px(N), \\
        & \text{s.t.} && x(k+1) = Ax(k) + Bu(k), \quad x(0) = x_0, \\
        & && x_{\text{min}} \leq x(k) \leq x_{\text{max}}, \quad k = 0, 1, \ldots, N, \\
        & && u_{\text{min}} \leq u(k) \leq u_{\text{max}}, \quad k = 0, 1, \ldots, N, 
    \end{aligned}
\end{equation}
where $N$ is the control horizon, and $Q$, $R$, and $P$ are the weight matrices for the states, inputs, and final states, respectively. It is assumed that $Q$ and $P$ are positive semidefinite, while $R$ is positive definite. The constraints $x_{\text{min}}$ and $x_{\text{max}}$ define the lower and upper bounds on the states, and $u_{\text{min}}$ and $u_{\text{max}}$ define the lower and upper bounds on the inputs.

The solution of the problem \eqref{eqn:mpc} yields the optimal control sequence $\mathbf{u}^*(x_0)= \{u^*(0; x_0), u^*(1; x_0), \ldots, u^*(N-1; x_0)\}$. At each sampling time, only the first control input $u^*(0; x_0)$ is applied to the system. Hence, the MPC control law is defined as $u_0 = \kappa_N(x_0)=u^*(0; x_0)$. At the next time step, the optimization problem \eqref{eqn:mpc} is resolved with updated state measurements. This receding horizon strategy ensures that the system is iteratively controlled while respecting constraints and minimizing the cost function.

\section{Exact Learning of MPC law using Oblique Decision Tree with Linear Predictions}
In this section, we first establish the equivalence between the Linear MPC control law defined by Problem (\ref{eqn:mpc}) and the \texttt{ODT-LP} model. Then, we present the training procedure and analyze the robustness of the proposed \texttt{ODT-LP}-based MPC controller.

\subsection{Equivalence of Linear MPC and Oblique Decision Trees}
The closed-form solution to Linear MPC problem \eqref{eqn:mpc} is a continuous, piece-wise affine function over a partitioned state space, formalized as follows:

\begin{corollary}(Adapted from \cite{bemporad_explicit_2002}, Corollary 2)\label{cor:empc}
The MPC control law $\kappa_N(x_0) = f(x_0), \ f:\mathbb{R}^n \to \mathbb{R}^m$, defined by the optimization problem \eqref{eqn:mpc}, is continuous and piecewise affine:
\begin{equation}
    u_0 = \kappa_N(x_0) = F^i x_0 + g^i, \ \text{if} \ H^i x_0 \leq l^i, \ i = 1, \ldots, N_{\text{mpc}},
\end{equation}
where the polyhedral sets $\{H^i x_0 \leq l^i\}, \ i = 1, \ldots, N_{\text{mpc}}$, form a partition of the given set of states.
\end{corollary}

Based on Corollary \ref{cor:empc}, we can state the following theorem: 

\begin{theorem}\label{the:dtmpc}
    The MPC control law defined by the optimization problem \eqref{eqn:mpc} is equivalent to a finite-depth oblique decision tree with linear predictions at the leaf nodes .
\end{theorem}

\begin{proof}
    From Corollary \ref{cor:empc}, the MPC control law $\kappa_N(x_0)$ is continuous and piecewise affine, with polyhedral regions $\{H^i x_0 \leq l^i\}$ and corresponding affine laws $\kappa_N(x_0) = F^i x_0 + g^i$. We construct an equivalent \texttt{ODT-LP} as follows: 

    \textit{Branch nodes:} Each branch node's split rule $a_t^\top x_0 < b_t$ encodes one row from $\{H^i x_0 \leq l^i\}$, which define one boundary of the polyhedral region. These split rules recursively partition the state space until terminal regions are reached. Since the number of polyhedral regions is finite, denoted as $N_{\text{mpc}}$, and each region has a finite number of boundaries, the total number of branch nodes is also finite. Thus, the resulting decision tree has finite depth.

    \textit{Leaf nodes:} Given the construction of branch nodes, the path to each leaf node corresponds to the boundaries of a single polyhedral region $\{H^i x_0 \leq l^i\}$. Consequently, each leaf node stores the associated affine control law $F^i x_0 + g^i$.
    
    Combining these elements completes the construction of the equivalent \texttt{ODT-LP} model.
\end{proof}

Notably, a single Linear MPC control law can correspond to multiple equivalent \texttt{ODT-LP} models. Furthermore, if we attempt to construct \texttt{ODT-LP} models by first computing the explicit form of the control law, the approach will be constrained by the same scalability problem of explicit MPC as mentioned in the introduction. Additionally, this method can lead to an unbalanced decision tree, resulting in disproportionate online computational time—for example, if all boundaries of a region are stored in right branch nodes, the decision path may become excessively deep. To address these challenges, we propose a data-driven framework that directly learns the MPC control law from sampled state-action pairs, bypassing explicit form computation. This approach yields an exact representation of the linear MPC control law for simpler problems and a high-performance approximation for complex controllers with a large number of regions.

\subsection{Procedure of the \texttt{ODT-LP}-based MPC} \label{sec:training}
\begin{figure}[!htb]
    \centering
    \includegraphics[width=\columnwidth]{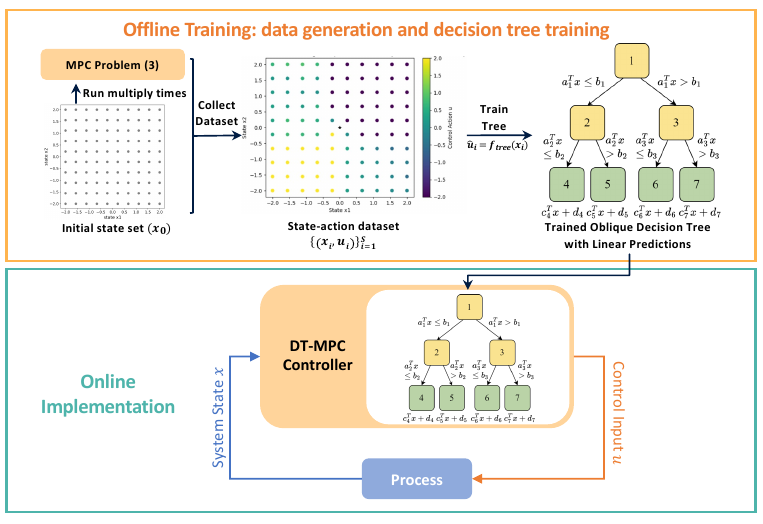}
    \caption{Procedure of the \texttt{ODT-LP}-based MPC}
    \label{fig:framework}
\end{figure}
As shown in Fig. \ref{fig:framework}, the training procedure of the \texttt{ODT-LP}-based MPC controller consists of two phases including offline training and online implementation. In the offline training phase, the dataset is generated by solving the MPC problem \eqref{eqn:mpc} across diverse initial states. The state-action dataset is then used to train the \texttt{ODT-LP} model in a supervised manner, where oblique splits are optimized to partition the state space into regions mirroring the MPC’s explicit solution, while leaf nodes encode precomputed affine control laws. The online implementation phase involves evaluating the trained \texttt{ODT-LP} model to compute the control action at each time step. 

\subsubsection{Data generation}
The data generation phase constructs a representative dataset for training the \texttt{ODT-LP} model. First, we use uniform grid sampling with a grid size $\delta_X$ to generate an initial set of states covering the system's operational envelope. Then, the MPC optimization problem (\ref{eqn:mpc}) is solved offline multiple times, starting from the sampled initial states $x_0$, to compute the corresponding optimal control actions $u_0 = \kappa_N(x_0)$. The resulting state-action pairs $\{(x_i, u_i)\}_{i=1}^S$ are stored in the dataset.

\subsubsection{Decision Tree Training} \label{sec:training_alogirthm}
The \texttt{ODT-LP} control law is trained using the generated dataset $\{(x_i, u_i)\}_{i=1}^S$. The control law is formulated as $\hat{u}_i = \kappa_{\texttt{ODT-LP}}(x_i) = f_{\text{tree}}(\mathbb{\theta}, x_i)$, where $\mathbb{\theta} = \{\mathbf{a}, \mathbf{b}, \mathbf{c}, \mathbf{d}\}$ comprises the split parameters ${a}_t \in \mathbb{R}^n$ and thresholds $b_t \in \mathbb{R}$ at branch nodes, as well as the affine coefficients ${c}_t \in \mathbb{R}^{n \times m}$ and biases $d_t \in \mathbb{R}^m$ at leaf nodes. The training objective minimizes the squared error over the training dataset $\mathcal{L} = \min_{\mathbb{\theta}} \sum_{i=1}^S \left\| \hat{u}_i - u_i \right\|_2^2.$

To enable gradient-based optimization, the problem is reformulated into a unified unconstrained form:
\begin{equation}\label{eq:training}
    \min_{\mathbb{\theta}} \sum_{i=1}^S \sum_{t \in \mathcal{T}_{L}} \left[ \prod_{j \in \mathcal{A}_t^l} I_{i,j} \prod_{j \in \mathcal{A}_t^r} \left(1 - I_{i,j}\right) \left\| u_i - ({c}_t^\top x_i + d_t) \right\|_2^2 \right],
\end{equation}
where $\mathcal{T}_{L}$ denotes the set of leaf nodes, $\mathcal{A}_t^l$ and $\mathcal{A}_t^r$ represent the subsets of branch nodes along the unique path to leaf $t$ where the left and right split conditions ${a}_j^\top x_i \leq b_j$ and ${a}_j^\top x_i > b_j$ are satisfied, respectively. The indicator $I_{i,j}=\mathds{1}(b_j-a_j^Tx_i \geq 0)$ equals $1$ if $x_i$ is split to the left at branch node $j$; otherwise, it equals $0$. The product terms $\prod_{j \in \mathcal{A}_t^l} I_{i,j}$ and $\prod_{j \in \mathcal{A}_t^r} (1 - I_{i,j})$ act as a binary selector: for each sample $x_i$, exactly one leaf node $t$ satisfies $\prod_{j \in \mathcal{A}_t^l} I_{i,j} \prod_{j \in \mathcal{A}_t^r} (1 - I_{i,j}) = 1$, ensuring only the corresponding affine model $({c}_t^\top x_i + d_t)$ contributes to the loss. To resolve the discontinuity of $I_{i,j}$, a scaled sigmoid function $\mathcal{S}(z) = \frac{1}{1 + e^{-\alpha z}}$ approximates the indicator as $I_{i,j} \approx \mathcal{S}(b_j - {a}_j^\top x_i)$, where $\alpha > 0$ controls the approximation sharpness. This smooth relaxation enables gradient propagation through the tree structure during training. Parameters $\mathbb{\theta}$ are updated via gradient descent to minimize the differentiable loss while preserving the interpretable hierarchy of splits and predictions. More details about the gradient-based optimization algorithm can be found in \cite{mao2024can}.

 \subsection{Robustness of \texttt{ODT-LP} based MPC}
\subsubsection{Error Analysis of the \texttt{ODT-LP} Policy}
We begin by stating an assumption on the learning accuracy of the \texttt{ODT-LP} model for the uniformly sampled dataset.

\begin{assumption}[Learning Accuracy]\label{ass:learning}
    Given a training set \(\{(x_i, u_i)\}_{i=1}^S\) covering the state space \(\mathcal{X}\), the \texttt{ODT-LP} learns a control policy \(\kappa_{\texttt{ODT-LP}}(x_i)\) that approximates the MPC control law \(\kappa_N(x_i)\) with bounded policy approximation error:
    \begin{equation}\label{eq:learning_error}
        |\kappa_{\texttt{ODT-LP}}(x_i) - \kappa_N(x_i)| \leq \delta_{DT} \quad \forall i \in \{1,\ldots,S\},
    \end{equation}
    where \(\delta_{DT} \geq 0\) denotes the maximum learning error. 
\end{assumption}

It is well-known that continuous and piecewise affine (CPWA) functions with finite pieces are Lipschitz continuous \cite{goujonStable2023}. By Corollary \ref{cor:empc}, the linear MPC control law \(\kappa_N(x)\) is a CPWA function over a finite polyhedral partition of the state space $\mathcal{X}$, thus Lipschitz continuous.
However, the trained \texttt{ODT-LP} policy could be discontinuous and include jumps across the boundaries between two adjacent leaf regions. Since the explicit form of the \texttt{ODT-LP} is known, we can characterize two key properties: its maximum local Lipschitz constant, $K_{DT} = \max_{t\in \mathcal{T}_L} ||\mathbf{c}_t||_2$, where $\mathbf{c}_t$ is the weight matrix in leaf node $t$; and its maximum discontinuity jump, $J_{\max} = \sup_{\mathbf{x} \in \mathcal{B}} |\kappa_{t_1}(\mathbf{x}) - \kappa_{t_2}(\mathbf{x})|$, where $\mathcal{B}$ is the set of boundaries between any two adjacent leaf nodes $t_1$ and $t_2$. Combining these properties with Assumption \ref{ass:learning}, we have:

\begin{lemma}[Policy Approximation Error Bound]\label{lem:control_error}
    Let \(K_{\text{max}}\) be the Lipschitz constant of the linear MPC control law \(\kappa_N(\mathbf{x})\) over \(\mathcal{X}\). The approximation error, denoted \(e_{DT}(\mathbf{x})\), between the \texttt{ODT-LP} policy and the ideal MPC law satisfies:
    \begin{equation}\label{eq:total_error}
        \begin{aligned}
            & | e_{DT}(\mathbf{x}) | = |\kappa_{\texttt{ODT-LP}}(\mathbf{x}) - \kappa_N(\mathbf{x})| \\
            & \quad \leq \delta_{DT} + J_{\max} + (K_{DT} + K_{\text{max}})\frac{\sqrt{n}\delta_X}{2}, \forall \mathbf{x} \in \mathcal{X},
        \end{aligned}
    \end{equation}
    where \(\delta_X\) is the grid size from the uniform grid sampling and $n$ is the dimension of the system state.
\end{lemma}
\begin{proof}
    For any point \(\mathbf{x} \in \mathcal{X}\), let its nearest sampled point in the training set be \(\mathbf{x}_{i, \text{near}}\). The maximum distance is bounded by \(|\mathbf{x} - \mathbf{x}_{i, \text{near}}| \leq \frac{\sqrt{n}\delta_X}{2}\). We decompose the total error using the triangle inequality:
    \begin{equation*}
        \begin{aligned}
            & |e_{DT}(\mathbf{x})| \leq |\kappa_{\texttt{ODT-LP}}(\mathbf{x}) - \kappa_{\texttt{ODT-LP}}(\mathbf{x}_{i, \text{near}})| \\
            & \quad + |\kappa_{\texttt{ODT-LP}}(\mathbf{x}_{i, \text{near}}) - \kappa_N(\mathbf{x}_{i, \text{near}})| + |\kappa_N(\mathbf{x}_{i, \text{near}}) - \kappa_N(\mathbf{x})|.
        \end{aligned}
    \end{equation*}
    We bound each of the three terms on the right-hand side. The first term accounts for the variation of the learned \texttt{ODT-LP} policy between a training point and an arbitrary nearby point; it is bounded by considering both the local affine variation and potential jumps between regions, yielding $|\kappa_{\texttt{ODT-LP}}(\mathbf{x}) - \kappa_{\texttt{ODT-LP}}(\mathbf{x}_{i, \text{near}})| \leq K_{DT}|\mathbf{x} - \mathbf{x}_{i, \text{near}}| + J_{\max}$. The second term is the learning error at a training point, which is bounded by Corollary \ref{ass:learning} as $|\kappa_{\texttt{ODT-LP}}(\mathbf{x}_{i, \text{near}}) - \kappa_N(\mathbf{x}_{i, \text{near}})| \leq \delta_{DT}$. Finally, the third term is the sampling error, bounded by the Lipschitz continuity of the MPC policy: $|\kappa_N(\mathbf{x}_{i, \text{near}}) - \kappa_N(\mathbf{x})| \leq K_{\text{max}}|\mathbf{x} - \mathbf{x}_{i, \text{near}}|$. Combining these bounds and substituting $|\mathbf{x} - \mathbf{x}_{i, \text{near}}| \leq \frac{\sqrt{n}\delta_X}{2}$ yields the final result.
\end{proof}

\subsubsection{Input-to-State Stability Analysis} 
Our analysis demonstrates that the control policy approximated by the \texttt{ODT-LP} model exhibits bounded approximation errors. Building on prior work by \cite{hertneckLearning2018, kumar_industrial_2021}, which establishes robustness guarantees for neural network (\texttt{NN}) controllers trained using nominal or robust MPC control laws, we treat the \texttt{ODT-LP} approximation error as an additive disturbance to the nominal MPC controller. By leveraging input-to-state stability (ISS) theory from \cite{sontagChar1995}, we extend these robustness results to \texttt{ODT-LP}-based controllers through the derived error bound.

Assuming nominal MPC stability holds, the optimal cost function $ V_N^0(x) $ satisfies $ c_1|x|^2 \leq V_N^0(x) \leq c_2|x|^2 $ and  $V_N^0(x^+) - V_N^0(x) \leq -c_1|x|^2 $ for constants $ c_1, c_2 > 0 $ and all $ x \in \mathcal{X}_N $, where $ \mathcal{X}_N:= \text{lev}_{Nl + \tau} \, V_N^0({x})$ is the region of attraction. Here, $N$ is the control horizon, $\tau>0$ is the level set parameter defining the terminal set $\mathbb{X}_f := \text{lev}_{\tau}\, x^\top P x$, and $l$ is a lower bound on the stage cost outside the terminal set, satisfying $ 0< l \leq \ell(x,u) = x^\top Q x + u^\top R u, \text{ for all } x\in\mathbb{R}^n\setminus \mathbb{X}_f$.

To establish closed-loop robustness under \texttt{ODT-LP} control, we adapt the framework from \cite{kumar_industrial_2021} using fundamental results from \cite{rawlingsModel2017}:

\begin{definition}[Robust Positive Invariance]
    A set $\mathcal{X} \subset \mathbb{R}^n$ is robustly positively invariant for $x^+ = f(x,w)$ if $x \in \mathcal{X},\ w \in \mathbb{W} \text{ implies } f(x,w) \in \mathcal{X}$.
\end{definition}

\begin{definition}[Input-to-State Stability (ISS)]
    The system $x^+ = f(x,w)$ is ISS in $\mathcal{X}$ if exists $\beta \in \mathcal{KL}$, $\sigma \in \mathcal{K}$ such that $|\phi(k;x,{w})| \leq \beta(|x|,k) + \sigma(|{w}|), \forall x \in \mathcal{X}, {w} \in \mathbb{W}$,
    where \(\phi(k;x,{w})\) denotes the solution at time \(k\).
\end{definition}

\begin{definition}[ISS-Lyapunov Function] \label{def:iss_lyap}
    $V: \mathcal{X} \to \mathbb{R}_{\geq 0}$ is an ISS-Lyapunov function for \(x^+ = f(x,w)\) if exists $\alpha_1,\alpha_2,\alpha_3 \in \mathcal{K}_\infty$, $\sigma \in \mathcal{K}$ that $\alpha_1(|x|) \leq V(x) \leq \alpha_2(|x|)$ and $V(f(x,w)) - V(x) \leq -\alpha_3(|x|) + \sigma(|w|)$.
\end{definition}

\begin{proposition}[Bounded Function Differences]\label{prop:bounded}
    Let \(\mathcal{C} \subset \mathcal{D} \subset \mathbb{R}^n\), where \(\mathcal{C}\) is compact and \(\mathcal{D}\) is closed. For any continuous function \(f: \mathcal{D} \to \mathbb{R}^n\), exists $\sigma \in \mathcal{K}_\infty$ such that:
    \begin{equation*}
        |f(x) - f(y)| \leq \sigma(|x - y|), \quad \forall x \in \mathcal{C},\ y \in \mathcal{D}.
    \end{equation*}
\end{proposition}

\begin{proposition}[Subadditivity of $\mathcal{K}$-Functions]\label{prop:subadd}
    \begin{equation*}
        \sigma\left(\sum_{i=1}^n a_i\right) \leq \sum_{i=1}^n \sigma(na_i), \text{for } \sigma \in \mathcal{K} \text{ and }a_i \geq 0.
    \end{equation*}
\end{proposition}

We now analyze the robustness of the perturbed closed-loop system under \texttt{ODT-LP} control:
\begin{equation}
    \begin{aligned}
        {x}^+ = A{x} + B\kappa_{\texttt{ODT-LP}}(\hat{x}) + w, \ \hat{x}= x + e,        
    \end{aligned}
\end{equation}
where $w$ is process disturbance and $e$ is state estimation error. 

\begin{theorem}[ISS of \texttt{ODT-LP} based MPC]\label{thm:odt-iss}
     For all \(0 < \rho \leq N l + \tau\), there exist constants \(\delta_1\), \(\delta_2\), \(\delta_3 > 0\), functions \(\beta(\cdot) \in \mathcal{KL}\), and \(\alpha_e(\cdot)\), \(\alpha_w(\cdot)\), \(\alpha_{DT}(\cdot) \in \mathcal{K}\), such that for all disturbance sequences satisfying \(|{e}_{k+1}| \leq \delta_1\), \(|{w}_k| \leq \delta_2\), and \(|e_{DT}(\hat{x})| \leq \delta_3\), and for all \(\hat{x}\) in the set \(\mathcal{S} := \text{lev}_\rho \, V_N^0(\hat{x})\), the closed-loop system satisfies $
    |\phi_d(k; \hat{x})| \leq \beta(|\hat{x}|, k) + \alpha_e(|{e}_{k+1}|) + \alpha_w(|{w}_k|) + \alpha_{DT}({e}_{DT, max})$, 
    where \({e}_{DT, max} = \max_{\hat{x} \in \mathcal{S}} |\kappa_{\texttt{ODT-LP}}(\hat{x}) - \kappa_N(\hat{x}) | \leq \delta_{DT} + J_{\max} + (K_{DT} + K_{\text{max}})\frac{\sqrt{n}\delta_X}{2}\).
\end{theorem}
\begin{proof}    
    \textbf{Step 1: Robust Positive Invariance of \(\mathcal{S}\).} 
    The perturbed system dynamics with ODT-LP approximation error at time $k$ are:
    \[
    \hat{x}^+ = A\hat{x} + B\kappa_N(\hat{x}) + B(\kappa_{\texttt{ODT-LP}}(\hat{x}) - \kappa_N(\hat{x})) + w - Ae + e^+.
    \]
    Let $d(k) = [e_{DT}(\hat{x})^\top B^\top, {w}^\top, {e}^\top, {{e}^{+\top}}]^\top$ denote the combined disturbances. By Proposition \ref{prop:bounded}, there is $\sigma_1 \in \mathcal{K}_\infty$ such that:
    \begin{equation*}
        \begin{aligned}
            |V_N^0(\hat{x}^+) - V_N^0(\hat{x}^+_{\text{nom}})| & \leq \sigma_1\left(\hat{x}^+ - \hat{x}^+_{\text{nom}}\right) \\
            & \leq \sigma_1\left(|B||{e}_{DT}(\hat{x})| + |w| + |A||e| + |e^+|\right) \\
            & \leq \sigma_1\left((|A| + 3)|d|\right) =: \sigma_2(|d|),
        \end{aligned}
    \end{equation*}
    where \(\hat{x}^+_{\text{nom}} = A\hat{x} + B\kappa_N(\hat{x})\) is the nominal state.
    
    Partition \(\mathcal{S}\) into two regions: (1) For \(\rho/2 \leq V_N^0(\hat{x}) \leq \rho\), with the cost decrease property of nominal MPC, we have $V_N^0(\hat{x}^+) \leq V_N^0(\hat{x}) - c_1|\hat{x}|^2 + \sigma_2(|d|).$ Since $|\hat{x}|^2 \geq \frac{\rho}{2c_2} $, we require $\sigma_2(|d|) \leq \frac{\rho c_1}{2c_2}$ for $V_N^0(\hat{x}^+)\leq \rho$; (2) For \(V_N^0(\hat{x}) \leq \rho/2\), we have $V_N^0(\hat{x}^+) \leq \frac{\rho}{2} + \sigma_2(|d|),$ requiring \(\sigma_2(|d|) \leq \frac{\rho}{2}\).
    
    Let \(\delta_{\max} = \sigma_2^{-1}\left(\min\left\{\frac{\rho c_1}{2c_2}, \frac{\rho}{2}\right\}\right)\). Since $c_2 \geq c_1$,  \(\delta_{\max} = \sigma_2^{-1}(\frac{\rho c_1}{2c_2})\). Set \(\delta_1 = \delta_{\max}/4\), \(\delta_2 = \delta_{\max}/4\), and \(\delta_3 = \delta_{\max}/(4|B|)\). These ensure \(|d| \leq \delta_{\max}\) and \(\mathcal{S}\) robustly invariant.

    \textbf{Step 2: ISS-Lyapunov Function.}
    The optimal cost \(V_N^0\) in the robustly invariant set $\mathcal{S}$ is lower and upper bounded and satisfies $
    V_N^0(\hat{x}^+) - V_N^0(\hat{x}) \leq -c_1|\hat{x}|^2 + \sigma_2(|d|)$. 
    This fulfills the ISS-Lyapunov conditions (Definition \ref{def:iss_lyap}), and implies ISS of the system \cite{rawlingsModel2017}. Through Definition \ref{def:iss_lyap} and applying Proposition \ref{prop:subadd} to decompose \(\sigma_2(|d|)\):
    \begin{equation*}
        \begin{aligned}
            |\phi_d(k; \hat{x})| \leq & \beta(|\hat{x}|, k) + \alpha_e(6|{e}_{k+1}|) 
            \\  + & \alpha_w(3|{w}_k|) + \alpha_{DT}(3|B|{e}_{DT, max}),
        \end{aligned}
    \end{equation*}
    where \(\alpha_e\), \(\alpha_w\), and \(\alpha_{DT}\) are \(\mathcal{K}\)-functions derived from \(\sigma_2\). 
\end{proof}

Theorem~\ref{thm:odt-iss} guarantees input-to-state stability (ISS) for the \texttt{ODT-LP} controller with three key implications: (1) The ISS bound decomposes the total error into learning error, sampling error, and disturbances; (2) The error bound \({e}_{DT, max} \leq \delta_{DT} + K_{\text{max}} \frac{\sqrt{n}\delta_X}{2}\) demonstrates that controller stability can be enhanced through improved decision tree training algorithms or refined grid resolution; (3) The \texttt{ODT-LP}-based MPC controller inherits MPC's inherent robustness to disturbances under bounded approximation errors.

\section{Numerical Experiments}
We evaluate the \texttt{ODT-LP}-based MPC controller on two benchmark linear systems from \cite{bemporad_explicit_2002, summers_multi_2011}, comparing it against the original MPC controller, a Neural-Network (\texttt{NN}) based controller, and a Random Forest (\texttt{RF}) based controller with the same number of parameters. All models were implemented in Python; the \texttt{ODT-LP} and \texttt{NN} models were trained using \texttt{Pytorch}, while MPC problems was solved with \texttt{Pyomo} and \texttt{Ipopt}. The \texttt{NN}s employ ReLU activation and were trained with the Adam optimizer using the same squared error loss function as the \texttt{ODT-LP}. To ensure robust testing, we generated a 250,000-sample dataset for each case study using uniform grid sampling. For standardized comparison, all experiments were conducted on a Linux workstation with Intel 6226R CPU, Nvidia A6000 GPU and 251GB RAM. However, the actual hardware requirements are substantially lower in our tests, as the \texttt{ODT-LP} training requires a maximum of 4GB of GPU RAM, and online predictions are executed on the CPU.

\subsection{Case Study 1: A Two-State System}
We consider the MPC problem of a two-state single-input single-output system from \cite{bemporad_explicit_2002} in the following form:
\begin{equation} \label{eqn:mpc_case1}
    \small
    \begin{aligned}
    & \min_{\mathbf{u}} && x(2)^\top Px(2) + \sum_{k=0}^{1} \left[x(k)^\top Ix(k) + 0.01u(k)^2\right], \\
    & \text{s.t.} && x(k+1) = \begin{bmatrix} 0.7326 & -0.0861 \\ 0.1722 & 0.9909 \end{bmatrix}x(k) + \begin{bmatrix} 0.0609 \\ 0.0064 \end{bmatrix}u(k), \\
    & && -2 \leq u(k) \leq 2, \quad k = 0, 1, \\
    & && x(0) = x(t), P = A^\top PA + I.
    \end{aligned}
\end{equation}

A depth-two \texttt{ODT-LP} model is trained on a dataset of 100,000 samples generated from a uniform grid with a grid size of 0.01, as visualized in Fig.~\ref{fig:dt_case1}. The learned model exhibits equivalence to the explicit MPC control law derived in~\cite{bemporad_explicit_2002}. To quantify closed-loop performance, we define the metric $\Lambda = \frac{1}{20} \sum_{t=1}^{20} \|x(t) - x_s(t)\|_Q^2,$ representing the average state tracking error over a 20-step horizon. The trained \texttt{ODT-LP} model is deployed to solve  problem~\eqref{eqn:mpc_case1} online, initialized from states in the standardized test set and simulated for 20 time steps. Performance results, summarized in Table~\ref{tbl:case1}, confirm that the \texttt{ODT-LP}-based controller achieves \emph{exact learning} of the optimal MPC control law, with no measurable deviation in state trajectories or control actions. 
\begin{figure}[!ht]
    \centering
    \includegraphics[width=0.9\columnwidth]{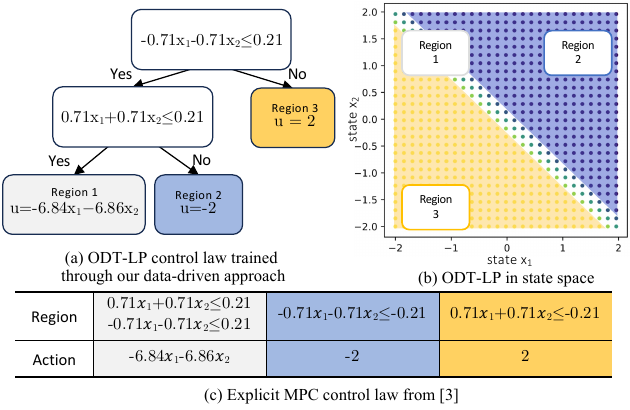}
    \vskip -0.1in
    \caption{Trained \texttt{ODT-LP} control law of case study 1.}
    \label{fig:dt_case1}
    \vskip -0.1in
\end{figure}

As for the online computation time, the \texttt{ODT-LP} model \textit{significantly reduces the online computation time} compared to the original MPC controller and the \texttt{NN}-based MPC controller. Comparing to the Explicit MPC, since this case study only has 3 regions, the average online computation time of the \texttt{ODT-LP} model is slightly higher than the Explicit MPC. However, the worst-case online computation time of the \texttt{ODT-LP} model is significantly lower than Explicit MPC, which is aligned with the analysis in Fig. \ref{fig:dt}. 
\begin{table}[!htp]
\vskip -0.1in
    \caption{Performance comparison on case study 1.}
    \label{tbl:case1}
    \setlength\tabcolsep{2pt}
    \renewcommand{\arraystretch}{1.2}
    \resizebox{\columnwidth}{!}{
    \begin{tabular}{c|cc|c|c|c|c}
        \hline
        Method & \multicolumn{1}{c|}{\begin{tabular}[c]{@{}c@{}}Training\\      Error\\      (RMSE)\end{tabular}} & \begin{tabular}[c]{@{}c@{}}Testing\\      Error\\      (RMSE)\end{tabular} & \begin{tabular}[c]{@{}c@{}}Closed-loop\\      Performance\\      ($\Lambda$)\end{tabular} & \begin{tabular}[c]{@{}c@{}}Performance \\      Loss to MPC\\      (\%)\end{tabular} & \begin{tabular}[c]{@{}c@{}}Online\\       Time\\      (average/worst, s)\end{tabular} & \begin{tabular}[c]{@{}c@{}}Average\\ Training\\ Time (s)\end{tabular} \\ \hline
        MPC & \multicolumn{1}{c|}{-} & - & 14.970 & - & 5.62E-2/3.90E-1 & - \\ \hline
        Explicit   MPC & \multicolumn{2}{c|}{3 regions} & 14.970 & 0.000 & 1.10E-5/3.84E-3 & - \\ \hline
        \begin{tabular}[c]{@{}c@{}}DT-MPC\\ (Depth = 2)\end{tabular} & \multicolumn{1}{c|}{1.78E-5} & 1.72E-4 & 14.970 & 0.000 & 2.32E-5/8.04E-4 & 158.01 \\ \hline
        \begin{tabular}[c]{@{}c@{}}NN-MPC\\ (2-3-3)\end{tabular} & \multicolumn{1}{c|}{3.15E-3} & 3.26E-3 & 14.972 & 0.013 & 9.00E-5/2.25E-2 & 20.52 \\ \hline
        RF-MPC & \multicolumn{1}{c|}{2.60E-03} & 6.12E-03 & 14.971 & 0.005 & 5.87E-3/1.45E-2 & 16.79 \\ \hline
    \end{tabular}
        }
\end{table}

\subsection{Case Study 2: A Four-State System}
We also consider a four-state MPC problem from \cite{summers_multi_2011}:
\begin{equation} \label{eqn:mpc_case2}
    \small
    \begin{aligned}
    & \min_{\mathbf{u}} && x(17)^\top Ix(17) + \sum_{k=0}^{16} \left[x(k)^\top Ix(k) + 0.2u(k)^2\right], \\
    & \text{s.t.} && x(k+1) = Ax(k) + Bu(k), x(0) = x(t), \\
    & && -0.2 \leq u(k) \leq 0.2, \quad k = 0, 1, \ldots, 16, \\
    & where && A = \begin{bmatrix}
        0.4035 & 0.3704 & 0.2935 & -0.7258 \\
        -0.2114 & 0.6405 & -0.6717 & -0.0420 \\
        0.8368 & 0.0175 & -0.2806 & 0.3808 \\
        -0.0724 & 0.6001 & 0.5552 & 0.4919
        \end{bmatrix}, \\
    & && B = \begin{bmatrix} 1.6124 & 0.4086 & -1.4512 & -0.6761 \end{bmatrix}^T.
    \end{aligned}
\end{equation}

\begin{table}[!htp]
    \caption{Performance comparison on case study 2.}
    \label{tbl:case2}
    \setlength\tabcolsep{2pt}
    \renewcommand{\arraystretch}{1.3}
    \resizebox{1.02\columnwidth}{!}{
        \begin{tabular}{c|cc|c|c|c|c}
            \hline
            Method & \multicolumn{1}{c|}{\begin{tabular}[c]{@{}c@{}}Training\\      Error\\      (RMSE)\end{tabular}} & \begin{tabular}[c]{@{}c@{}}Testing\\      Error\\      (RMSE)\end{tabular} & \begin{tabular}[c]{@{}c@{}}Closed-loop\\      Performance\\      ($\Lambda$)\end{tabular} & \begin{tabular}[c]{@{}c@{}}Performance \\      Loss to MPC\\      (\%)\end{tabular} & \begin{tabular}[c]{@{}c@{}}Online\\       Time\\      (average/worst, s)\end{tabular} & \begin{tabular}[c]{@{}c@{}}Average\\ Training\\ Time (s)\end{tabular} \\ \hline
            MPC & \multicolumn{1}{c|}{-} & - & 218.10 & 0 & 6.72E-2/3.02E+1 & - \\ \hline
            Explicit   MPC & \multicolumn{2}{c|}{58,227 regions} & 218.10 & 0 & 7.26E-1/1.83E+0 & 15337.58 \\ \hline
            \begin{tabular}[c]{@{}c@{}}DT-MPC\\ (Depth = 8)\end{tabular} & \multicolumn{1}{c|}{1.39E-02} & 1.78E-02 & 219.11 & 0.46 & 6.40E-5/1.24E-4 & 2257.74 \\ \hline
            \begin{tabular}[c]{@{}c@{}}NN-MPC\\ (40-40-20-10)\end{tabular} & \multicolumn{1}{c|}{6.59E-03} & 4.72E-03 & 218.68 & 0.26 & 1.63E-4/1.80E-3 & 83.79 \\ \hline
            RF-MPC & \multicolumn{1}{c|}{7.97E-03} & 2.16E-02 & 219.86 & 0.81 & 7.78E-3/1.91E-2 & 6.27 \\ \hline
        \end{tabular}
        }
\end{table}

The state space is uniformly sampled with {100,000 points and 0.02 grid size} to train a depth-eight \texttt{ODT-LP} model. As summarized in Table~\ref{tbl:case2}, the \texttt{ODT-LP}-based controller closely approximates the MPC control law, exhibiting a minimal performance loss of {0.46\%} relative to the original MPC controller. The closed-loop response of the \texttt{ODT-LP}-based controller, illustrated in Fig.~\ref{fig:response_case2}, further confirms its ability to track the original MPC's behavior with high accuracy. Note that the explicit MPC control law consists of 58,227 regions, requiring at least a depth-16 decision tree for exact representation. While deeper trees could, in theory, achieve an \textit{exact representation}, increasing depth introduces trade-offs: deeper trees are inherently {less interpretable} and {more computationally demanding} to train, necessitating larger datasets and advanced solvers. These challenges underscore the advantage of \texttt{ODT-LP} to provide a \textit{high-fidelity approximation} even with constrained tree depths, effectively balancing {prediction accuracy}, {computational efficiency}, and control law interpretability. The \texttt{ODT-LP}-based MPC controller requires significantly less online computation time than the original, explicit, and \texttt{NN}-based MPC controllers, which is consistent with the analysis in Fig.~\ref{fig:dt}.

\begin{figure}[!htp]
    \centering
    \includegraphics[width=0.9\columnwidth]{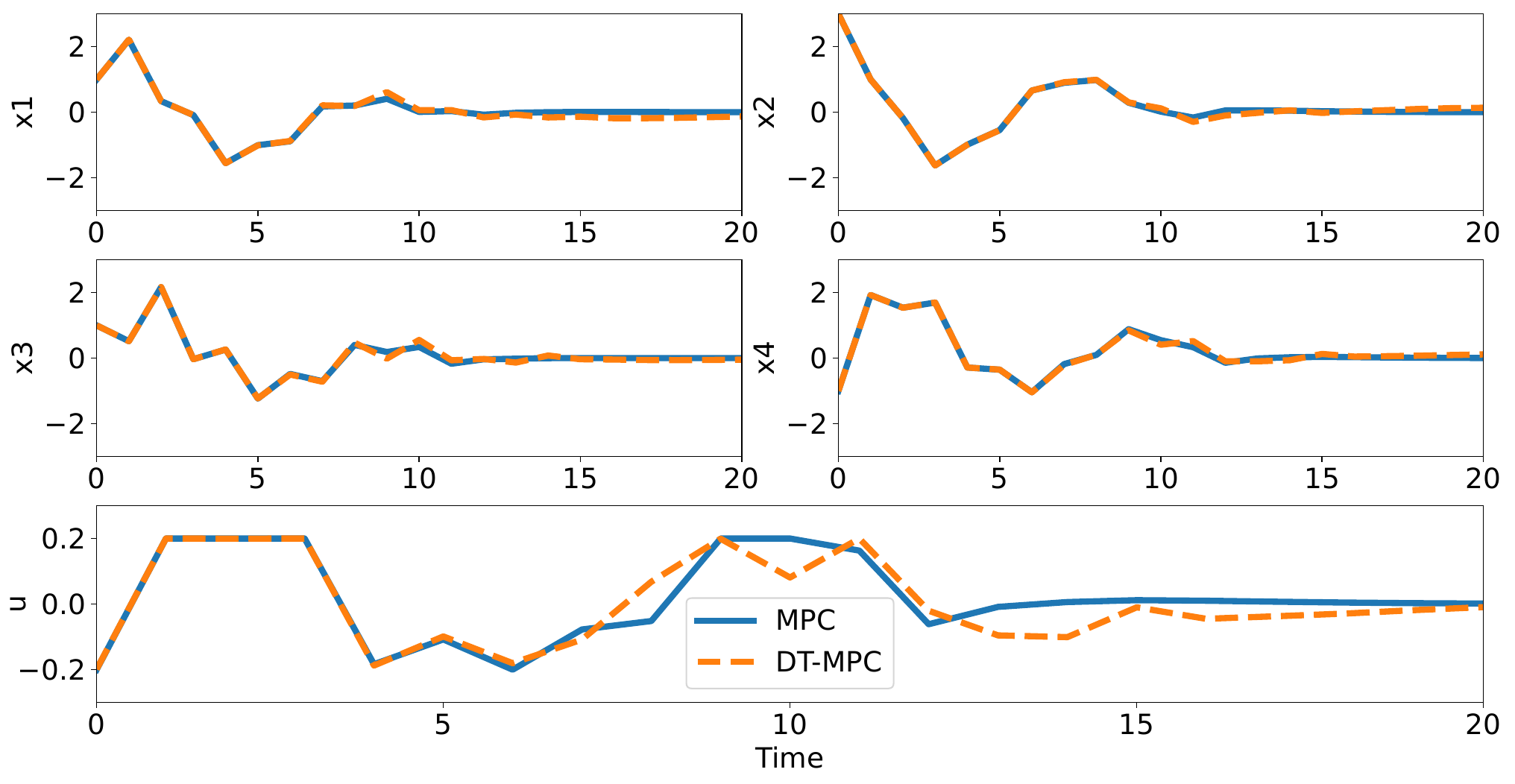}
    \caption{Closed-loop response of \texttt{ODT-LP}-based MPC controller.}
    \label{fig:response_case2}
\end{figure}

\section{Conclusion} 
In conclusion, we have introduced a novel data-driven framework for the exact learning of Linear Model Predictive Control law using oblique decision trees with linear predictions. We established the equivalence between \texttt{ODT-LP} and the Linear MPC control law, presented the training procedure with the uniform grid sampling method for data generation and the gradient-based solver for decision tree training, and analyzed the robustness of the \texttt{ODT-LP}-based MPC controller. Benchmark evaluations on linear systems demonstrate that the proposed approach effectively approximates Linear MPC control laws while reducing online computation time by orders of magnitude. The proposed approach addresses critical limitations of existing methods: it avoids the exponential complexity of explicit MPC and the interpretability challenges of neural networks. The transparent architecture—featuring verifiable hierarchical splits and interpretable linear leaf models—enables controller verification and debugging. Future work will focus on extending the proposed approach to nonlinear systems, as well as integrating with reinforcement learning. To enhance real-world applicability where collecting a uniformly gridded dataset may not be feasible, we will also investigate advanced sampling strategies. Methodologies like sparse grids can reduce the required number of training samples, while reinforcement learning can potentially eliminate the need for uniform sampling by allowing the model to learn from online interactions.


\bibliographystyle{ieeetr}
\bibliography{dtmpc_reference}

\end{document}